\documentclass[11pt,reqno]{amsart}

\usepackage{enumerate}
\usepackage{amsmath, amssymb, amsthm}
\usepackage{mathrsfs}
\usepackage{esint}
\usepackage[dvipsnames]{xcolor}
\usepackage{mathtools}
\usepackage{hyperref}
\usepackage{bm}
\usepackage{tikz}
\usepackage{pgfplots}
\usepackage{float,kotex,graphicx}
\pgfplotsset{compat=1.18}

\numberwithin{equation}{section}
\newtheorem{theorem}{Theorem}[section]
\newtheorem{lemma}[theorem]{Lemma}

\newtheorem{remark}[theorem]{\bf{Remark}}

\newtheorem{definition}[theorem]{Definition}

\theoremstyle{remark}
\theoremstyle{definition}

\def \data {./data}

\newcommand{\emphTitle}[1]{\medskip\noindent\textbf{#1.}\quad}

\def \L {L_2}
\def \H {H^2_2}
\def \C {\cC^{0.8}}

\newcommand\bR{\mathbb{R}}

\newcommand\cC{\mathcal{C}}

\newcommand\cF{\mathcal{F}}

\newcommand\cL{\mathcal{L}}

\newcommand\cbrk{\text{$]$\kern-.15em$]$}}
\newcommand\opar{\text{\,\raise.2ex\hbox{${\scriptstyle
|}$}\kern-.34em$($}}
\newcommand\cpar{\text{$)$\kern-.34em\raise.2ex\hbox{${\scriptstyle |}$}}\,}

\newcommand\ep{\varepsilon}

\begin{document}

\title[]
{A Multiplicative Neural Network Architecture:\\
Locality and Regularity of Approximation
}

\author{Hee-Sun Choi}
\address{Department of Artificial Intelligence and Robotics, 
Sejong University, 
Seoul, Republic of Korea}
\email{heesunchoi@sejong.ac.kr}

\author{Beom-Seok Han$^*$}
\address{Department of Mathematics,
Sungshin Women’s University,
Seoul, Republic of Korea}
\email{b\_han@sungshin.ac.kr}
\thanks{$^*$ Corresponding author, This work was supported by the Sungshin Women’s University Research Grant of 2024}

\subjclass[2020]{}



\begin{abstract}
We introduce a multiplicative neural network architecture in which multiplicative interactions constitute the fundamental representation, rather than appearing as auxiliary components within an additive model.
We establish a universal approximation theorem for this architecture and analyze its approximation properties in terms of locality and regularity in Bessel potential spaces.

To complement the theoretical results, we conduct numerical experiments on representative targets exhibiting sharp transition layers or pointwise loss of higher-order regularity.
The experiments focus on the spatial structure of approximation errors and on regularity-sensitive quantities, in particular, the convergence of Zygmund-type seminorms.
The results show that the proposed multiplicative architecture yields residual error structures that are more tightly aligned with regions of reduced regularity and exhibit more stable convergence in regularity-sensitive metrics.

These results demonstrate that adopting a multiplicative representation format has concrete implications for the localization and regularity behavior of neural network approximations, providing a direct connection between architectural design and analytical properties of the approximating functions.
\end{abstract}

\maketitle

\section{Introduction}
\label{sec:intro}

From a mathematical perspective, the theory of neural networks is fundamentally grounded in the approximation of functions. (\cite{cybenko1989approximation,hornik1991approximation,leshno1993multilayer,pinkus1999approximation})
For example, \cite{cybenko1989approximation} introduced the approximation framework based on a class of univariate functions represented by a single-layer feedforward neural network of the form
\begin{equation}
    \label{form of standard NN}
G(x) \;=\; \sum_{j=1}^{n} \alpha_j\, \sigma(w_j\cdot x + b_j),
\end{equation}
where $x = (x^{1},x^{2},\dots,x^{m}) \in [0,1]^m$ denotes the input vector, $\cdot$ is the dot product, $\sigma: \bR \to \bR$ is a fixed non-linear activation function applied to each neuron in the hidden layer, $w_j \in \bR^m$ and $b_j \in \bR$ are the weight vector and bias of the $j$-th hidden neuron, and $\alpha_j \in \bR$ is the corresponding output weight, respectively.

A fundamental property of single-layer feedforward networks is their ability to \textit{universally approximate} continuous functions.  The following theorem, due to Cybenko \cite{cybenko1989approximation}, establishes that a single hidden layer and a suitable activation can approximate any continuous function on a compact domain with arbitrary precision.
\begin{theorem}[Universal Approximation -- Cybenko] 
\label{Cybenko theorem}
The set of functions of the form \eqref{form of standard NN} is dense in $C([0,1]^m)$, the space of continuous real-valued functions on the unit cube. That is, for any $f\in C([0,1]^m)$ and any $\ep>0$, there exists a function $G$ of type \eqref{form of standard NN} with finitely many neurons $n$ such that 
\begin{equation*}
\sup_{x\in [0,1]^{m}}|G(x)-f(x)|<\ep.
\end{equation*}
\end{theorem}
Theorem \ref{Cybenko theorem} implies that standard multilayer perceptrons (MLPs) with a single hidden layer and a suitable activation function are capable of approximating any continuous function defined on a bounded domain.

However, the structural design of MLPs often gives rise to nonlocal effects when employed for function approximation. Here, nonlocality refers to the mismatch between the spatial extent of neuron activations and the localized nature of the target features to be approximated. To illustrate this phenomenon, we consider the one-dimensional case first. Let $f(x):[0,1]\to\bR$ be a given continuous function. Using the network form introduced in Theorem \ref{Cybenko theorem}, we construct a simple function $s(x)$, which is a piecewise constant approximation of $f(x)$ from below.
\begin{center}
\includegraphics[scale=0.5]{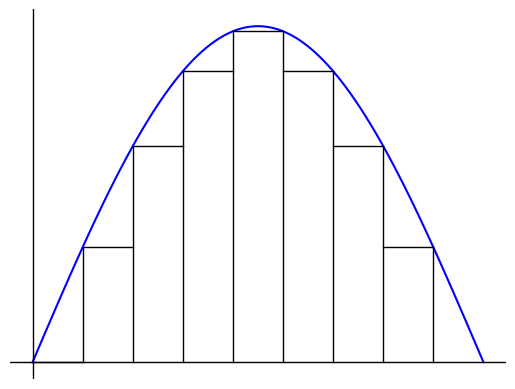}\\
\scriptsize Figure 1. A simple function $s(x)$ approximating a continuous function $f(x)$ from below.
\end{center}
The function $s(x)$ can be represented as
\begin{equation*}
s(x) = \sum_{j = 1}^{n} \alpha_{j}\, \sigma(w_{j}x + b_{j}),
\end{equation*}
where $\sigma(\cdot)$ denotes an indicator function. The function $s(x)$ is a linear combination of indicator functions, each supported on a small subinterval. If the indicator $\sigma(\cdot)$ is replaced by a smooth approximation--such as a mollified bump function--serving as an activation function in a neural network, then each term $\sigma(w_{j}x + b_{j})$ can be interpreted as a single neuron whose activation is effectively concentrated on a localized region. Intuitively, constructing the approximation $s(x)$ requires $n$ such neurons, each corresponding to one mollified indicator.

Now, we turn to the two-dimensional case. Let $f(x,y):[0,1]^{2}\to\bR$ be a continuous function. In this setting, an analogous simple function $s(x,y)$ is used to approximate $f(x,y)$, defined by
\begin{equation*}
s(x,y) = \sum_{j = 1}^{n}\alpha_{j}\sigma(w_{j}^{1}x+w_{j}^{2}y+b_{j}).
\end{equation*}
Here, the input to each neuron is a linear combination of the input variables, $x$, $y$, and the bias term. As a result, the region where a single neuron $\sigma(w_{1j}x+w_{2j}y+b_{j})$ active corresponds to a strip bounded by two parallel lines in the input domain. Figure 2 illustrates the support of such a neuron as the shaded region bounded by two parallel lines.
\begin{center}
\includegraphics[scale=0.5]{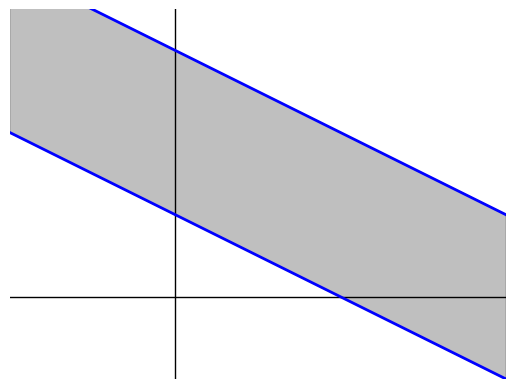}\\
\scriptsize Figure 2. The nonzero region (gray) of a single neuron $\sigma(w_{j}^{1}x + w_{j}^{2}y + b_j)$ in two dimensions.
\end{center}

Furthermore, the activation region resulting from the linear combination of two neurons is illustrated in Figure 3.
\begin{center}
\includegraphics[scale=0.5]{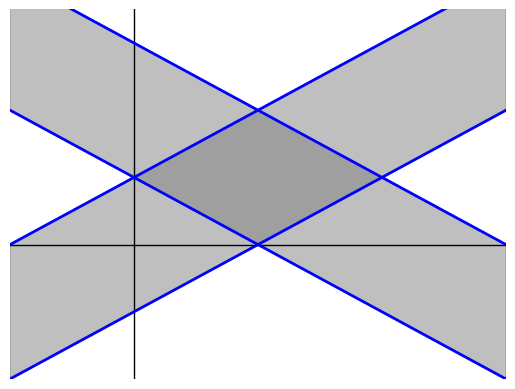}\\
\scriptsize Figure 3. Overlapping activation regions of two neurons. The darker region indicates the area of joint activation.
\end{center}
Observe that the intersection of the individual activation regions forms the primary domain contributing to the approximation. However, due to the global nature of each neuron, redundant activation outside the target region also occurs.

Therefore, the nonlocal nature of MLPs arises from their architecture. For instance, when approximating a compactly supported bump function or an indicator function in a multi-dimensional domain, a standard MLP may require a large number of neurons whose activation regions significantly extend beyond the region of interest. This leads to redundant computation and inefficiency in representing localized structures, thereby highlighting the intrinsic nonlocality of the architecture. From a geometric viewpoint, localization is naturally achieved by intersections of activation
regions, whereas standard additive MLPs can only realize such intersections indirectly
through superpositions of globally supported neurons.

To detour the nonlocal property, we propose a new version of the MLP architecture, referred to as the \textit{Multiplicative MLP} (MMLP), which introduces \textbf{multiplicative interactions} into the model’s neurons. As an representative example, consider a function of the form
\begin{equation}
\label{Multiplicative MLP example}
\sigma(w^{1}x+b^{1})\sigma(w^{2}y+b^{2}),
\end{equation}
where $w^{1},w^{2},b^{1},b^{2}\in \bR$.
The activation region corresponding to \eqref{Multiplicative MLP example} is illustrated in Figure 4.
\begin{center}
\includegraphics[scale=0.5]{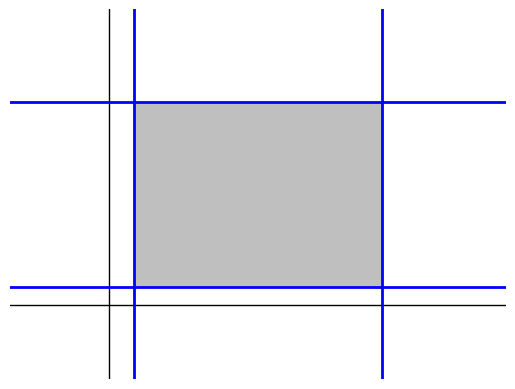}\\
\scriptsize Figure 4. Activation region of a multiplicative unit, forming a rectangular domain in the input space.
\end{center}
Note that the activation domain in this case forms a rectangular region, as it is defined by the simultaneous activation of two independent one-dimensional neurons. In contrast to standard MLPs, the basic building block of the approximation is localized without redundant overlaps.

Neural network architectures involving multiplicative interactions have appeared in earlier
work, most notably in the study of higher--order and product--type networks.
In particular, sigma--pi networks and sum--of--product architectures replace additive hidden
units by products of one--dimensional responses, and their approximation properties have been
investigated in a variety of functional settings; see, for example,
\cite{lenze1994sigmapi,lickli2003spsnn,lincsli2000sopnn,long2007uniform,luo2000lp}.
In this literature, approximation capability is typically understood in the sense of density
or global approximation in classical function spaces, and the resulting approximation
classes may extend beyond those generated by purely additive architectures.

The perspective taken in the present work is of a different nature.
Rather than focusing on global approximation capability alone, we emphasize the geometric and
analytic structure induced by multiplicative units themselves.
A central observation is that a product of one--dimensional activations induces spatial
localization through intersection of activation regions, already at the level of a single
multiplicative unit.
This property makes it possible to interpret a single multiplicative unit as a localized
mollifier--type kernel, in contrast to additive architectures, where comparable localization
typically requires superposition of many globally supported units.

This mollifier interpretation provides the link to approximation theory in function spaces
with regularity.
In particular, it allows one to employ the approximation framework of Bessel potential spaces
$H_p^\gamma(\mathbb{R}^m)$ and, via Sobolev embedding, to obtain corresponding results in the
associated Zygmund spaces.
The results presented below make this connection precise and show that multiplicative
architectures admit a regularity--aware mode of approximation that is not available for
standard additive MLPs.

In this paper, informally speaking, we show that neural networks equipped with multiplicative units can approximate functions not only universally, but also locally and with controlled regularity, in the sense that they are dense in Sobolev–Bessel and Zygmund-type function spaces.

To formulate our approximation result in a setting that simultaneously captures smoothness and is compatible with Fourier-analytic techniques underlying multiplicative structures, we recall the definition of Bessel potential spaces and Zygmund spaces. For more information, see \cite{grafakos2009modern,krylov2024lectures,triebel1983theory}

\begin{definition}[$L_p$ spaces]
For $1\le p<\infty$,  the space $L_p(\bR^m)$ consists of all (equivalence classes of) measurable functions
$u:\bR^m\to\mathbb{R}$ such that
\[
\|u\|_{L_p(\bR^m)}
:= \left(\int_{\bR^m} |u(x)|^p\,dx\right)^{1/p} < \infty.
\]
\end{definition}

\begin{definition}[Bessel Potential Space]
For $p>1$ and $\gamma \in \bR$, let $H_{p}^{\gamma}(\bR^{m})$ denote the class of tempered distributions $f$ on $\bR^m$ such that
\begin{equation*}
\| f \|_{H_{p}^{\gamma}(\bR^{m})} 
:= \left\| (1-\Delta)^{\gamma/2} f  \right\|_{L_p(\bR^m)} 
= \left\| \cF^{-1}\left[ \left(1+|\cdot|^2\right)^{\gamma/2}\cF[f]\right] \right\|_{L_p(\bR^m)}
<\infty,
\end{equation*}
Here, $\cF$ denotes the $m$-dimensional Fourier transform and, $\cF^{-1}$ denotes the $m$-dimensional inverse Fourier transform. 
\end{definition}

\begin{remark}[Relation to classical Sobolev spaces]
When the regularity $\gamma$ is a nonnegative integer, the Bessel potential space
$H_p^\gamma(\mathbb{R}^m)$ coincides with the classical Sobolev space
$W_{p}^{\gamma}(\mathbb{R}^m)$.
More precisely, for $\gamma\in\{0,1,2,\dots\}$,
\[
H_p^\gamma(\mathbb{R}^m)=W_{p}^{\gamma}(\mathbb{R}^m),
\]
with equivalence of norms.
\end{remark}

\begin{definition}[Zygmund Space]
\label{zygmund space}
For $k = 0,1,2,\dots$ and $\alpha \in (0,1]$, and a multi-index $\beta$ with $|\beta|=k$, set
\[
[f]_{\cC^{k+\alpha}(\bR^{m})} := 
\sup_{x \in \bR^m,\, h \neq 0}
\frac{\big| D^\beta f(x+h) + D^\beta f(x-h) - 2D^\beta f(x) \big|}{|h|^\alpha},
\]
and
\[
|f|_{\cC^k(\bR^{m})} := \sum_{j=0}^k \sup_{|\beta|=j}\, \sup_{x \in \mathbb{R}^m} |D^\beta f(x)|,
\qquad
|f|_{\cC^{k+\alpha}(\bR^{m})} := |f|_{\cC^k(\bR^{m})} + [f]_{\cC^{k+\alpha}(\bR^{m})}.
\]

We define $\cC^{k+\alpha}(\mathbb{R}^m)$ as the set of bounded continuous functions on $\mathbb{R}^d$ such that the finite norm 
$|f|_{\cC^{k+\alpha}(\bR^{m})} < \infty$.
\end{definition}

\begin{remark}
When $\alpha\in(0,1)$, $\cC^{k+\alpha}(\mathbb{R}^m)$ coincides with the classical H\"older space $C^{k,\alpha}(\bR^{m})$.
When $\alpha=1$, however, the Zygmund space $\cC^{k+1}(\bR^{m})$ properly contains $C^{k,1}(\bR^{m})$,
thus providing the natural endpoint of the H\"older scale.
\end{remark}

\begin{lemma}
\label{embedding lemma}
Let $p\in (1,\infty)$ and $\gamma>m/p$. Then 
\begin{equation*}
H_p^\gamma(\mathbb{R}^m) \hookrightarrow \cC^{\gamma - \frac{m}{p}}(\mathbb{R}^m).
\end{equation*}
In other words, if $\gamma - m/p = \eta+\nu$ for some $\eta=0,1,\cdots$ and $\nu\in(0,1]$, then for any  $k\in\{ 0,1,\cdots,\eta \}$, we have
\begin{equation*} 
| D^k f |_{C(\bR^{m})} + | D^\eta f |_{\cC^\nu(\bR^{m})} \leq N \| f \|_{H_{p}^\gamma(\bR^{m})}.
\end{equation*}

\end{lemma}

Now, we present the universal approximation theorem employing multiplication.

\begin{theorem}
\label{main theorem}
Let $\gamma\geq0$ and $p\in (1,\infty)$. Assume that $\sigma:\bR \to \bR$ is a nontrivial function such that $\sigma\in H_{p}^{\gamma}(\bR)$. Suppose $S$ is a set of functions $F:\bR^{m}\to \bR$ of the form
\begin{equation}
\label{eq:approximation function}
F(x) = \sum_{j = 1}^{n}\alpha_{j} \prod_{i = 1}^{m}\sigma\left( w^{T}_{ij} x + b_{ij} \right),
\end{equation}
where 
$n\in \{1,2,\dots,\}$, $\alpha_{j}\in\bR$,  $w_{ij}$, $x\in \bR^{m}$, $b_{ij}\in \bR$, and $
\begin{bmatrix}
w_{1j}~w_{2j}~\cdots~w_{mj}
\end{bmatrix}^{T}$
is an invertible matrix. Then, $S$ is dense in $H_p^{\gamma}(\bR^{m})$. In other words, given any $f\in H^{\gamma}_p(\bR^{m})$ and $\ep>0$, there exists the function $F$ of the form \eqref{eq:approximation function} such that
\begin{equation*}
\| F - f \|_{H_p^{\gamma}(\bR^{m})} \leq \ep.
\end{equation*}
In particular, if $\gamma - \frac{m}{p} >0 $, then $f\in \cC^{\gamma-\frac{m}{p}}(\bR^{m})$ and
\begin{equation*}
\| F - f \|_{\cC^{\gamma-\frac{m}{p}}((\bR^{m}))} \leq \ep,
\end{equation*}
where $\cC^{\gamma-\frac{m}{p}}(\bR^{m})$ is the Zygmund space.
\end{theorem}

\begin{remark}[Invertibility of weight matrices]
The invertibility of the matrix $(w_{ij})$ ensures that the associated affine transformation
of the input space is nondegenerate and that each multiplicative unit acts on a full-dimensional
region. This assumption is made for technical clarity and does not restrict the generality
of the approximation result.
\end{remark}

\begin{remark}[Meaning and role of the parameters in Theorem~\ref{main theorem}]
\label{rem:meaning-parameters}
The parameters appearing in Theorem~\ref{main theorem} are interpreted as follows.

\begin{enumerate}[(i)]
    \item $n$ denotes the number of multiplicative units (or neurons) in the hidden layer.
    The density statement is asymptotic in $n$, in the sense that a target function
    $f \in H_p^\gamma(\mathbb{R}^m)$ can be approximated as $n \to \infty$.
    No explicit convergence rate with respect to $n$ is asserted.

    \item $m$ is the dimension of the input space $\mathbb{R}^m$.
    The multiplicative structure in \eqref{eq:approximation function} combines $m$
    one-dimensional factors, so that each unit acts on a genuinely $m$-dimensional region.
    The theorem holds for all $m \ge 1$, and the dependence on the dimension enters only
    through the Sobolev embedding exponent $\gamma - m/p$.

    \item $\gamma \ge 0$ represents the order of smoothness measured in the Bessel potential space $H_p^\gamma$. Larger values of $\gamma$ correspond to higher-order regularity, and the theorem asserts density in the sense of $H_{p}^{\gamma}(\bR^{m})$. 

    \item $p \in (1,\infty)$ is the integrability exponent in the Bessel potential spaces.
    Together with $\gamma$ and $m$, the parameter $p$ determines the level of pointwise regularity obtained via Sobolev embedding. 
\end{enumerate}
\end{remark}

\begin{remark}[Sobolev embedding and approximation of regular functions]
Lemma~\ref{embedding lemma} explains how approximation results in Bessel potential spaces can be interpreted
in terms of pointwise regularity through the Sobolev embedding theorem.
For example, for
\[
\gamma = 2,\qquad m = 2,\qquad p = 2,
\]
we have $\gamma - m/p = 1$, and hence
\[
H^2_2(\mathbb{R}^2) \hookrightarrow \cC^1(\mathbb{R}^2),
\]
where $\cC^1$ is understood in the Zygmund sense. Therefore, Lemma \ref{embedding lemma} and Theorem~\ref{main theorem} imply that functions $f\in H_{p}^{\gamma}(\bR^{m})$ have $\cC^{\gamma - m/p}(\bR^{m})$ modification and the modification of $f$ can be approximated by the proposed multiplicative architecture via approximation in $\cC^{\gamma - m/p}(\bR^{m})$.

More generally, the embedding exponent $\gamma - m/p$ presents the role of the integrability
parameter $p$. For fixed regularity $\gamma$ and dimension $m$, larger values of $p$ lead to embeddings into spaces with higher Zygmund regularity. Thus, approximation in $H^\gamma_p(\bR^m)$ may be viewed as approximation of increasingly regular functions as $p$ grows.

In Section~\ref{sec:numeric}, we numerically investigate this particular case
$(\gamma,m,p)=(2,2,2)$ by employing an $H^2_{2}(\bR^{2})$-type training objective and evaluating the resulting approximations using Zygmund seminorms. This provides a concrete numerical illustration of how approximation in $H_2^2(\bR^{2})$ translates into control of pointwise regularity through Sobolev embedding.
\end{remark}

\begin{remark}[Training in $H_p^\gamma$ and evaluation via Zygmund seminorms]
\label{rem:training-eval-zyg}
The approximation result of Theorem~1.8 is stated in the Bessel space
$H_p^\gamma(\mathbb{R}^m)$.
By the Sobolev embedding of Lemma~1.7, approximation in $H_p^\gamma(\mathbb{R}^m)$ implies control of the corresponding Zygmund regularity $\mathcal{C}^{\gamma-m/p}(\mathbb{R}^m)$ whenever $\gamma>m/p$.

Accordingly, training and evaluation may be carried out in different norms.
Training based on an $H_p^\gamma$-type objective enforces control of weak derivatives in an $L_p$ sense, while evaluation using Zygmund seminorms provides a measure of the induced pointwise regularity. This separation is consistent with the functional-analytic structure underlying Theorem~1.8 and Lemma~1.7 and is adopted in the numerical experiments of Section~3.
\end{remark}

\begin{remark}[Bounded-domain interpretation and cutoff activations]
\label{rem:bounded-domain-activation}
In numerical experiments, both the input domain and the available data are bounded.
Accordingly, the approximation problem may be viewed as posed on a bounded domain
$\Omega \subset \mathbb{R}^m$ rather than on the whole space.

From this bounded-domain perspective, activation functions that do not belong to
$H_p^\gamma(\mathbb{R})$ globally may be interpreted locally.
More precisely, one may replace an activation $\sigma$ by a cutoff-modified function of the
form $\sigma(x)\,\zeta(x)$, where $\zeta\in C_c^\infty(\mathbb{R})$ is identically equal to one
on the relevant range of inputs.
The resulting function belongs to $H_p^\gamma(\mathbb{R})$ for appropriate values of
$(p,\gamma)$ and coincides with the original activation on $\Omega$.

This observation shows that, in the bounded-domain setting of the numerical experiments, the distinction between global and local function space frameworks is not essential.
\end{remark}

\begin{remark}[Choice of activation functions in numerical experiments]
\label{rem:activation-choice}
The activation functions employed in this work are selected according to their compatibility
with approximation in Bessel potential spaces and with the evaluation of approximation errors
in Zygmund spaces.
Table~\ref{tab:activation-choice} summarizes the regularity properties of the activations
considered in this setting.

Classical smooth activations such as the sigmoid and $\tanh$ functions satisfy the
assumptions of Theorem~\ref{main theorem} (see Remark \ref{rem:bounded-domain-activation}) and are compatible with approximation in
$H_p^\gamma(\mathbb{R}^m)$.
In particular, their smoothness ensures that approximation in $H_p^\gamma(\mathbb{R}^m)$ is well defined
and that higher-order regularity can be meaningfully assessed through Sobolev embedding.

Gaussian bump activations provide an alternative class of smooth functions with rapid decay.
Although not compactly supported, they belong to $H_p^\gamma(\mathbb{R})$ for all
$\gamma \ge 0$ and exhibit strong localization properties.
When used in multiplicative architectures, products of one-dimensional Gaussian factors
give rise to activation regions that are localized in the input space, thereby
highlighting the geometric structure of multiplicative units.

The use of $\tanh$ and Gaussian activations thus allows us to isolate architectural effects
on error localization while remaining fully consistent with the regularity framework of
Theorem~\ref{main theorem} and with the Zygmund-type evaluations considered in the numerical
experiments.
\end{remark}

\begin{table}[ht]
\centering
\caption{Activation functions and compatibility with the regularity framework.}
\label{tab:activation-choice}
\begin{tabular}{p{3.8cm} p{3.5cm} p{6.2cm}}
\hline
Activation type & Examples & Compatibility with the present framework \\
\hline
Smooth activations
& Sigmoid, $\tanh$
& Smooth and bounded; belong to $H_p^\gamma(\Omega)$ for all bounded domain $\Omega$ and $\gamma\geq 0$ and are
compatible with approximation in $H_p^\gamma$ and with evaluation in Zygmund spaces \\
\hline
Piecewise linear activations
& ReLU
& Have first-order weak derivatives; belong to $H^1_{p,\mathrm{loc}}(\Omega)$ but not to
$H_p^\gamma(\Omega)$ for $\gamma\geq 2$ \\
\hline
Localized smooth activations
& Gaussian bump
& Smooth and rapidly decaying; belong to $H_p^\gamma(\mathbb{R})$ for all $\gamma\ge 0$; yield
localized activation regions when combined with multiplicative units \\
\hline
\end{tabular}
\end{table}

\begin{remark}[Interpretability of multiplicative units]
Each factor of the form $\sigma(w^T x + b)$ is supported on a half-space or a strip in the
input domain, depending on the activation function.
For a multiplicative unit of the form
\[
\prod_{i=1}^{m} \sigma(w_i^T x + b_i),
\]
the support is given by the intersection of the corresponding activation regions.
This intersection defines a localized subset of the domain, such as a rectangle,
parallelepiped, or, more generally, a polytope.

By contrast, additive MLP units have activation regions that extend globally as strips.
Multiplicative units are therefore associated with intrinsically localized regions, and the
contribution of each block is confined to a specific subset of the input space.
\end{remark}

\begin{remark}
In \cite{cybenko1989approximation}, Cybenko considered function classes of the form
\[
\sum_{j=1}^{N}\alpha_j \sigma(U_j x + y_j),
\]
where $\sigma$ is the indicator function of a fixed rectangle with sides parallel to the coordinate axes in $\mathbb{R}^m$, $U_j$ is an orthogonal $m\times m$ matrix, and $y_j\in\mathbb{R}^m$.
By means of the Wiener--Tauberian theorem, it was shown that such functions are dense in
$L^1(\mathbb{R}^m)$.
This result establishes density via translations and rotations of rectangular indicator
functions.

The setting considered in the present work differs in two respects. From \eqref{eq:approximation function}, if $\sigma:\bR\to\bR$ is an indicator function of an interval, the basic building block reduces to a MMLP unit of the form
\[
\prod_{i=1}^m \sigma(w_i x_i + b_i), \quad w_i,\,b_i\in \bR
\]
which is supported on a rectangular, or more generally polyhedral, subset of the input
domain.
Second, the activation function $\sigma$ is allowed to be any nontrivial function in
$H_p^\gamma(\mathbb{R})$ with $p>1$ and $\gamma\in\bR$, rather than a fixed indicator function.
\end{remark}




\section{Theoretical Analysis}
\label{sec:theory}

In this section, we provide the proof of Theorem \ref{main theorem}. Additionally, several remarks are presented.

\begin{proof}[Proof of Theorem \ref{main theorem}]
First we consider the case $\sigma\in C_{c}^{\infty}(\bR)$ is nontrivial. Let $S$ be a set of functions of the form \eqref{eq:approximation function}. 

First we show that $S$ is a linear subspace of $H_p^\gamma(\bR^{m})$. If we choose $f$ and $g$ in $S$, then for any $a,b\in\bR$, we have
\begin{equation*}
\begin{aligned}
af(x) + bg(x)
&= a\sum_{j = 1}^{n_{1}}\alpha_{j} \prod_{i = 1}^{m}\sigma\left( w^{T}_{ij} x + b_{ij} \right) + b\sum_{j = 1}^{n_{2}}\bar\alpha_{j} \prod_{i = 1}^{m}\sigma\left( \bar w^{T}_{ij} x + \bar b_{ij} \right) \\
&= \sum_{j = 1}^{n_{1} + n_{2}}c_{j} \prod_{i = 1}^{m}\sigma\left( \tilde w^{T}_{ij} x + \tilde b_{ij} \right),
\end{aligned}
\end{equation*}
where $c_{j} = a\alpha_{j}1_{1\leq j \leq n_{1}} +b\bar \alpha_{j-n_{1}}1_{n_{1} < j \leq n_{2}}$, $\tilde w_{ij} = w_{ij}1_{1\leq j \leq n_{1}}+\bar w_{i(j-n_{1})}1_{n_{1} < j \leq n_{2}}$, and $\tilde b_{ij} = b_{ij}1_{1\leq j \leq n_{1}}+\bar b_{i(j-n_{1})}1_{n_{1} < j \leq n_{2}}$

Thus, $S$ is a linear space. Next, observe that $S\subset H_p^\gamma(\bR^{m})$. Indeed,
\begin{equation*}
\begin{aligned}
\left\|   \prod_{k = 1}^{m} \sigma(w_{k}^{T}\cdot + b_{k})   \right\|_{H_{p}^{\gamma}(\bR^{m})}
& \leq N\left\|   \prod_{k = 1}^{m} \sigma(M_{k})   \right\|_{H_{p}^{\gamma}(\bR^{m})} 
\leq N\left\|    \sigma   \right\|_{H_{p}^{\gamma}(\bR)}^{m}<\infty,
\end{aligned}
\end{equation*}
where $w_{k}\in \bR^{m}$, $b_{k}\in \bR$, and 
\begin{equation*}
M_{k}(x) = x^{k}
\end{equation*} 
Note that the first inequality follows from the fact that the linear transformations only changes the norm constant times. In the case of second inequality, notice that
\begin{equation*}
\begin{aligned}
  (1-\Delta)^{\gamma/2}\prod_{k = 1}^{m} \sigma(x^{k})  
&= \int_{\bR^{m}} e^{2\pi i \xi \cdot x}(1+|\xi|^{2})^{\gamma/2}\int_{\bR^{m}} e^{-2\pi i \xi \cdot y} \prod_{k = 1}^{m}\sigma(y^{k}) dy d\xi  \\
&= \int_{\bR^{m}} e^{2\pi i \xi \cdot x} \mu(\xi^{1},\xi^{2},\cdot,\xi^{m}) \hat g(\xi) d\xi \\
&= \cL g(x)
\end{aligned}
\end{equation*}
where 
\begin{equation*}
\mu(\xi^{1},\xi^{2},\cdot,\xi^{m}) := \frac{(1+|\xi|^{2})^{\gamma/2}}{\prod_{k = 1}^{m} (1+|\xi^{k}|^{2})^{\gamma/2}}, \quad g(x) = \prod_{k = 1}^{m}\left[ \left( \left( 1-\frac{d^{2}}{dx^{2}} \right)^{\gamma/2}\sigma \right)(x^{k}) \right]
\end{equation*}
and 
\begin{equation*}
\cL(\cdot) = \cF^{-1}(m \cF(\cdot)).
\end{equation*}
is a zeroth-order pseudo-differential operator. Thus, 
\begin{equation*}
\begin{aligned}
\left\|   \prod_{k = 1}^{m} \sigma(M_{k})   \right\|_{H_{p}^{\gamma}(\bR^{m})} 
&= \left\|   (1-\Delta)^{\gamma/2} \prod_{k = 1}^{m} \sigma(M_{k})   \right\|_{L_{p}(\bR^{m})} \\
&= \left\| \cL g   \right\|_{L_{p}(\bR^{m})} \\
&\leq N\left\|  g   \right\|_{L_{p}(\bR^{m})} \\
&= N\left\| \prod_{k = 1}^{m}\left[ \left( (1-D^{2})^{\gamma/2}\sigma \right)(M_{k}) \right]   \right\|_{L_{p}(\bR^{m})} \\
&= N\left\| \left[ \left( (1-D^{2})^{\gamma/2}\sigma \right) \right]   \right\|_{L_{p}(\bR)}^{m} \\
&= N\left\| \sigma   \right\|_{H_{p}^{\gamma}(\bR)}^{m}.
\end{aligned}
\end{equation*}

Now, we prove that the space $S$ is dense in $H_p^\gamma(\bR^{m})$. To apply a proof by contradiction, assume that there is $f_0\in H_p^\gamma(\bR^{m})\setminus \bar S$. Then, by the Hahn-Banach theorem, there is a bounded linear functional $\Lambda$ on $H_p^\gamma(\bR^{m})$ such that $\Lambda f = 0$ for $f\in \bar S$ and 
\begin{equation}
\label{contradiction point}
\Lambda f_0\neq 0.
\end{equation}
Then, by the Riesz representation theorem, there exists nonzero $h\in H_q^{-\gamma}(\bR^{m})$ with $q = p/(p-1)$ such that for any $f\in H_p^\gamma(\bR^{m})$,
\begin{equation*}
\Lambda f = (f,h).
\end{equation*}
For example, $(f,h) = \int_{\bR^{m}} f(x)h(x) dx$ if $\gamma = 0$.

For $\ep>0$ and $y = (y^{1},y^{2},\dots,y^{m})\in\bR^{m}$, set
\begin{equation}
\label{mollifier}
\xi (x) := \frac{1}{\| \sigma \|_{L_1(\bR)}^{m}}\prod_{i = 1}^{m}\sigma(x^{i})\quad\text{and}\quad 
\xi_{\ep}(x) := \ep^{-m}\xi\left( x/\ep \right).
\end{equation}
Since $\sigma\in C_c^\infty(\bR)$ is nontrivial, $\| \sigma \|_{L_{1}(\bR)}\neq 0$, and thus $\xi$ is well defined. 
Observe that the function
$x\mapsto \xi_\varepsilon(y-x)$ belongs to $S$ for each $y\in\mathbb{R}^m$. Indeed, for each $y\in\mathbb{R}^m$,
\[
\xi_\varepsilon(y-x)
=\varepsilon^{-m}\prod_{i=1}^m \sigma\!\left(\frac{y_i-x_i}{\varepsilon}\right)
\]
is of the form \eqref{eq:approximation function} with $n=1$, where 
\[
\begin{bmatrix}
w_{1j}~w_{2j}~\cdots~w_{mj}
\end{bmatrix}^{T} = W =\mathrm{diag}\!\left(\frac1\varepsilon,\dots,\frac1\varepsilon\right),
\]
and $b_i=y_i/\varepsilon$. Since $W$ is invertible, $\xi_\varepsilon(y-\cdot)\in S$ for all $y\in\mathbb{R}^m$.

Define
\begin{equation*}
h_{\ep}(y) 
:= \left(\xi_{\ep}(y - \cdot),h\right).
\end{equation*}
Then, $h_\varepsilon(y)
=(\xi_\varepsilon(y-\cdot),h)=0$ for all $y\in\mathbb{R}^m$,
and hence $h_\varepsilon\equiv0$.
On the other hand, since 
\begin{equation}
\label{mollification}
\|h-h_\varepsilon\|_{H_q^{-\gamma}(\mathbb{R}^m)} \to 0
\end{equation}
(e.g., see \cite[Theorem 13.9.2]{krylov2024lectures}) and $h_\varepsilon\equiv0$ for all $\varepsilon>0$, it follows that
\[
\|h\|_{H_q^{-\gamma}(\mathbb{R}^m)}=\|h-h_\varepsilon\|_{H_q^{-\gamma}(\mathbb{R}^m)}\to0,
\]
and hence $h=0$.
This contradicts \eqref{contradiction point} and therefore $S$ is dense in
$H_p^\gamma(\mathbb{R}^m)$.

Now we consider the case $\sigma\in H_{p}^{\gamma}(\bR)$. Let $\ep>0$. Since $C_{c}^{\infty}(\bR)$ is dense in $H_{p}^{\gamma}(\bR)$, there exists $\{\sigma_{\ell}\}\subset C_{c}^{\infty}(\bR)$ such that 
\begin{equation}
\label{sigma approximation}
\|\sigma_{\ell} - \sigma\|_{H_{p}^{\gamma}(\bR)}\leq 1/\ell
\end{equation}
By the previous step, there exists $F_{\ell}$ such that
\begin{equation*}
F_{\ell} = \sum_{j = 1}^{n}\alpha_{j} \prod_{i = 1}^{m}\sigma_{\ell}\left( w^{T}_{ij} x + b_{ij} \right)
\end{equation*}
and 
\begin{equation*}
\| F_{\ell} - f \|_{H_{p}^{\gamma}(\bR^{m})} \leq \ep.
\end{equation*}
Note that
\begin{equation*}
\left\|F - F_{\ell} \right\|_{H_{p}^{\gamma}(\bR^{m})} \leq N/\ell^m.
\end{equation*}
Indeed,
\begin{equation*}
\begin{aligned}
\left\|F - F_{\ell} \right\|_{H_{p}^{\gamma}(\bR^{m})}
&=\left\| \sum_{j = 1}^{n}\alpha_{j} \left(\prod_{i = 1}^{m}\sigma\left( w^{T}_{ij} \cdot + b_{ij} \right) - \prod_{i = 1}^{m}\sigma_{\ell}\left( w^{T}_{ij} \cdot + b_{ij} \right) \right)\right\|_{H_{p}^{\gamma}(\bR^{m})} \\
&\leq N\| \sigma - \sigma_{\ell} \|_{H_{p}^{\gamma}(\bR)}^{m} \\
&\leq N/\ell^m.
\end{aligned}
\end{equation*}
Thus,
\begin{equation*}
\| F - f \|_{H_{p}^{\gamma}(\bR^{m})} 
\leq \|F - F_{\ell}\|_{H_{p}^{\gamma}(\bR^{m})}
+ \|F_{\ell} - f\|_{H_{p}^{\gamma}(\bR^{m})} \leq N/\ell^m+\ep.
\end{equation*}
Since $\ell=1,2,\dots$ is arbitrary, the theorem is proved.
\end{proof}

\begin{remark}[Single-block localization via a mollifier-type kernel]
\label{rem:single-block-mollifier}
A key feature of the multiplicative architecture is that \emph{a single multiplicative unit}
already generates a localized mollifier-type kernel. In the proof of Theorem \ref{main theorem}, $\xi$ is a mollifer and thus \eqref{mollification} holds. 
Moreover, for each $y\in\bR^{m}$,
\[
\xi_\varepsilon(y-x)
=\varepsilon^{-m}\prod_{i=1}^{m}\sigma\!\left(\frac{y_i-x_i}{\varepsilon}\right)
\]
is exactly of the form \eqref{eq:approximation function} with $n=1$ and an invertible weight
matrix (diagonal scaling).
In particular, \emph{one} multiplicative block produces a translated and rescaled
mollifier-type bump, hence a genuinely localized building block at the level of a single
neuron.

This should be contrasted with standard additive MLPs, where a single unit
$\sigma(w^Tx+b)$ is supported on a strip/half-space and thus cannot represent a localized
approximate identity without superposition of many globally supported units.
\end{remark}

\begin{remark}[Geometric interpretation and locality of multiplicative architectures]
The distinction between standard additive MLPs and MMLP architectures is
reflected in the geometry of their basic building blocks.
In an additive MLP, each unit of the form $\sigma(w^T x + b)$ is supported on a strip-like
or half-space region of the input domain.
Consequently, such units have global extent in directions orthogonal to $w$, and spatial
localization can only be achieved through superposition of multiple globally supported
functions.

In a multiplicative architecture, several factors are combined through products.
Each factor $\sigma(w_i^T x + b_i)$ is supported on a strip-like region, and their product
is supported on the intersection of these regions.
This intersection yields a localized subset of the input domain.
Depending on the activation function, the resulting support may take the form of a
rectangular region, a smooth bump, or, more generally, a polytope.

This geometric difference is reflected in approximation behavior.
As observed in the numerical experiments of Section~\ref{sec:numeric}, additive MLPs
exhibit nonlocal propagation of approximation errors near boundaries and geometric
singularities, whereas multiplicative architectures confine errors to neighborhoods of the
target support or boundary.
This behavior is consistent with the density result of Theorem~\ref{main theorem} and
illustrates the role of geometric structure in approximation by multiplicative models.
\end{remark}

\begin{remark}
One may also consider approximation by functions of the form
\begin{equation}
\label{eq:k-not-m}
f(x)
=
\sum_{j=1}^{n}\alpha_j
\prod_{i=1}^{k}\sigma\!\left(w_{ij}^T x + b_{ij}\right),
\end{equation}
where the number of multiplicative factors is $k$, which need not coincide with the spatial
dimension $m$.
For simplicity, we restrict our attention to the case $\gamma=0$.

\begin{enumerate}[(i)]
\item
If $k<m$, then the class of functions in \eqref{eq:k-not-m} is dense in $L_p(\Omega)$ for any
bounded domain $\Omega\subset\mathbb{R}^m$.
For example, if $k=1$ and $m=2$, then functions of the form
$f(x)=\sigma(wx+b)$ belong to $L_{p}(\Omega)$ for any bounded $\Omega$, since
\[
\int_\Omega |\sigma(wx+b)|^p\,dx\,dy
\le N\|  \sigma\|_{L_{p}(\bR)}^{p} < \infty,
\]
provided $\sigma\in L_{p}(\bR)$.

\item
If $k>m$, then density in $L_p(\mathbb{R}^m)$ can be obtained under the additional assumption
$\sigma\in L_p(\mathbb{R})\cap L_\infty(\mathbb{R})$.
For example, when $k=2$ and $m=1$, one has
\[
\int_{\mathbb{R}} |\sigma(w_1 x)|^p\,|\sigma(w_2 x)|^p\,dx
\le \|\sigma\|_{L^\infty(\mathbb{R})}^p
\int_{\mathbb{R}} |\sigma(w_1 x)|^p\,dx < \infty,
\]
which ensures integrability of the multiplicative block.
The proof of Theorem~\ref{main theorem} can then be adapted to this setting.
\end{enumerate}

In practical numerical settings, it is natural to restrict attention to bounded domains,
since the available data are finite.
On bounded domains, continuity of $\sigma$ is sufficient to guarantee 
$L_{p}(\Omega)$ for $1<p\le\infty$.
\end{remark}







\section{Numerical Experiments}
\label{sec:numeric}

\subsection{Experimental Setup and Objectives}
\label{sec:numeric:setup}

The purpose of the numerical experiments is to illustrate the locality phenomenon predicted by the theoretical analysis in the previous sections, rather than to demonstrate quantitative superiority of one architecture over another.
In particular, the experiments are designed to examine whether the structural interpretation of multiplicative units as localized approximation, especially developed in Section \ref{sec:theory}, is reflected in the network predictions and the Zygmund seminorms. 

Throughout this section, the term {\it{locality}} refers to the concentration of approximation errors in regions where the target function exhibits large gradients or reduced regularity.

\emphTitle{Target functions and singular geometries}
We consider two representative target functions on a fixed bounded domain $\Omega = [-1,1]^{2} \subset \mathbb{R}^2$. Both are chosen to probe distinct geometric features associated with localized approximation difficulty.
\begin{itemize} 
    \item 
    Mollified Circle (sharp transition layer). 
    The first target is defined by
        \begin{align}
        f_{\mathrm{circle}}(x,y)
        = \frac{1}{2}
        \left(
        1 + \tanh\!\left(\frac{r_0 - \sqrt{x^2 + y^2}}{\varepsilon}\right)
        \right)
        \end{align}
    with fixed parameter $r_0 = 0.5$ and $\varepsilon = 0.05$. This function is smooth on $\Omega$, but exhibits a thin annular region 
    $\{ (x,y) \in \mathbb{R}^2 \mid \sqrt{x^2 + y^2} \approx r_0 \}$
    in which the gradient is large. Although the function is smooth, its sharp transition layer creates localized approximation difficulty reminiscent of discontinuous interfaces.
    \item 
    Mollified Cone (point-type reduced regularity).
    The second target is given by
        \begin{align}
        f_{\mathrm{cone}}(x,y)
        = \left(
        \max\!\left(1 - \sqrt{x^2 + y^2},\, 0\right)
        \right)^{\beta}
        \end{align}
    where $\beta = 1.8$. 
    Unlike the previous example, this target exhibits a loss of second-order regularity at the origin, while remaining continuous and once differentiable. This allows us to examine whether similar localization behavior persists when the region of difficulty collapses to a single point.
\end{itemize}   
The two examples therefore test locality under qualitatively different geometric scenarios: a sharp transition layer and a point of reduced regularity.
This distinction is consistent with the discussion in Section \ref{sec:theory}, which emphasizes that localization arises from the architectural structure rather than from the specific geometry of the target.

\emphTitle{Training data and sampling}
Training data consist of 50,000 points sampled uniformly at random from the domain $\Omega$. No validation or test sets are introduced: the experiments focus exclusively on approximation behavior on a fixed domain, rather than on extrapolation or generalization. This choice is consistent with the approximation-theoretic perspective adopted throughout the paper.

\emphTitle{Architectures and fair comparison}
We compare two classes of neural networks: a \emph{standard additive multilayer perceptron (MLP)} and a \emph{multiplicative multilayer perceptron (MMLP)}.
Following the notation introduced in Remark~\ref{rem:meaning-parameters}, we consider networks with input dimension \(m\), scalar output, and a single hidden layer. In that remark, the number of hidden units is denoted by \(n\). In the present numerical section, however, we distinguish between additive neurons and multiplicative blocks to avoid notational ambiguity.
With this convention, an MLP with \(n\) hidden neurons has a total of \((m+2)n+1\) trainable parameters, whereas an MMLP with \(n_b\) multiplicative blocks (denoted simply by \(n\) in Remark~\ref{rem:meaning-parameters}) has \((2m+1)n_b+1\) trainable parameters. Since all target functions considered in the numerical experiments are two-dimensional (\(m=2\)), the corresponding parameter counts reduce to \(4n+1\) for the MLP and \(5n_b+1\) for the MMLP.
To ensure a fair comparison between the two architectures, the values of \(n\) and \(n_b\) are chosen so that the total number of trainable parameters is {\it{exactly}} matched. All remaining training settings are kept identical across architectures, including the optimizer, learning rate, number of training iterations, initialization, and activation function. Both tanh and Gaussian activations are tested to assess robustness with respect to the choice of nonlinearity.
The specific network architectures and the values of \(n\) and \(n_b\) used in the experiments are summarized in Table \ref{tab:arch_config}.

\emphTitle{Loss functions and regularity-aware training}
As part of the experimental setup, we consider two loss functions, which are used to probe the regularity-sensitive behavior suggested by the theoretical analysis. The purpose of considering two loss functions is not to optimize performance, but to examine whether the regularity-sensitive quantities highlighted in the theoretical analysis exhibit meaningful behavior under training. 
\begin{itemize}
\item {\textbf{Baseline $\L$ loss.}}
    As a baseline, we consider the standard $\L$ loss
    \begin{align}
        \mathcal{L}_{\L}(F)
        &= \int_{\Omega} \left| F(x) - f(x) \right|^2 \, dx,
    \end{align}
    where $f$ denotes the target function and $F$ the network output.
    Up to a constant scaling factor, this loss coincides with the standard mean-squared error (MSE) commonly used in numerical implementations.
    Minimizing the \(\L\) loss controls the error only in an averaged sense and does not impose any explicit constraint on higher-order regularity or pointwise oscillations.
    For this reason, the \(\L\) loss serves as a natural baseline for examining how the spatial distribution of approximation errors depends on the network architecture.
\item {\textbf{$\H$-type loss with Laplacian penalty.}}
    To investigate regularity-aware training, we additionally consider an $\H$-type loss of the form
    \begin{align}
    \mathcal{L}_{\H}(F)
    &=
    \int_{\Omega} \left| F(x) - f(x) \right|^2 \, dx
    + \lambda
    \int_{\Omega} \left| \Delta F(x) - \Delta f(x) \right|^2 \, dx,
    \end{align}
    where $\lambda > 0$ is a fixed penalty parameter.
    This choice is motivated by the approximation results formulated in Bessel potential spaces.
    In particular, the analysis in Section \ref{sec:theory} establishes approximation in \(H^\gamma_p\) and relates such control to pointwise regularity through the Sobolev--Zygmund embedding.
    From this viewpoint, incorporating second-order derivative information into the loss allows us to directly control an \(\H\)-type quantity and to examine how this control influences the behavior of Zygmund-type seminorms in the numerical experiments.
\end{itemize}
The use of both $\L$ and $\H$-type loss serves a diagnostic purpose.
Since the \(\L\) loss does not include any derivative-based penalty, the spatial distribution of approximation errors is primarily determined by the network architecture and the geometry of the target function.
In contrast, the \(\H\)-type loss incorporates second-order derivative information through a Laplacian term. This makes it possible to examine how the inclusion of such information influences the behavior of regularity-sensitive metrics, including the convergence of Zygmund-type seminorms.

Importantly, we do not claim that $\H$-type training necessarily leads to uniformly smaller global errors. Rather, the numerical experiments are designed to check whether the qualitative regularity behavior suggested by the theoretical analysis —specifically, the connection between higher-order control and Zygmund regularity— can be observed empirically.

Throughout the experiments, the penalty parameter is fixed to $$\lambda = 10^{-2}.$$
This value is chosen to balance the two terms in the loss and to provide a consistent regularity signal during training. No attempt is made to tune $\lambda$ or to identify an optimal value, as such an investigation lies beyond the scope of the present work. Our goal is solely to assess the qualitative robustness of the observed phenomena under the inclusion of a higher-order penalty. 
The effect of the loss function choice on regularity-sensitive behavior is discussed further in Section \ref{sec:numeric:zygmund}.

\subsection{Error localization and spatial structure of approximation errors}
\label{sec:numeric:error_localization}

In this subsection, we examine the {\emph{spatial structure}} of approximation errors produced by additive (MLP) and multiplicative (MMLP) architectures for the target functions. 
Rather than comparing global approximation errors or claiming superiority of one architecture over another, the analysis focuses on whether the locality mechanisms identified in the theoretical analysis of previous sections are reflected in the geometry of the residual errors.

Throughout this subsection, all results correspond to training with the $\L$ loss. (The influence of the loss function choice is discussed separately in Section \ref{sec:numeric:zygmund}. )

\subsubsection{Mollified circle: boundary-localized error structure}
We first consider the mollified circle example, which is smooth on $\Omega$ but exhibits a thin annular region of large gradient. 
Figure~\ref{fig:circle:absErr} shows the spatial distributions of the absolute error \(|F(x)-f(x)|\), $x \in \Omega$ for both architectures, using tanh and Gaussian activations.

The error distributions exhibit the following qualitative features.
\begin{itemize}
\item Boundary localization.
    For both MLP and MMLP, the dominant error concentrates near the transition layer \(\sqrt{x^2 + y^2} \approx r_0\).
    This behavior is driven by the geometry of the target function and is therefore common to both architectures.
\item Residual structure away from the boundary.
    Away from the transition layer, the two architectures exhibit qualitatively different residual patterns.
    The MLP typically produces more spatially dispersed residual errors, with non-negligible error spreading over a broader region of the domain.
    By contrast, the MMLP yields residual errors that remain more localized near the transition layer and display a more organized spatial pattern.
\item Robustness with respect to activation functions.
    The qualitative difference in residual structure persists for both tanh and Gaussian activations, indicating that the observed behavior is not an artifact of a particular nonlinearity.
\end{itemize}
These observations are consistent with the interpretation developed in Section \ref{sec:theory}, where multiplicative units give rise to localized building blocks through coordinate-wise interactions.
In the numerical experiments, this is reflected not in the elimination of residual errors, but in a more structured spatial distribution of those errors.

\begin{figure}[ht]\centering
    \includegraphics[width=0.7\linewidth]{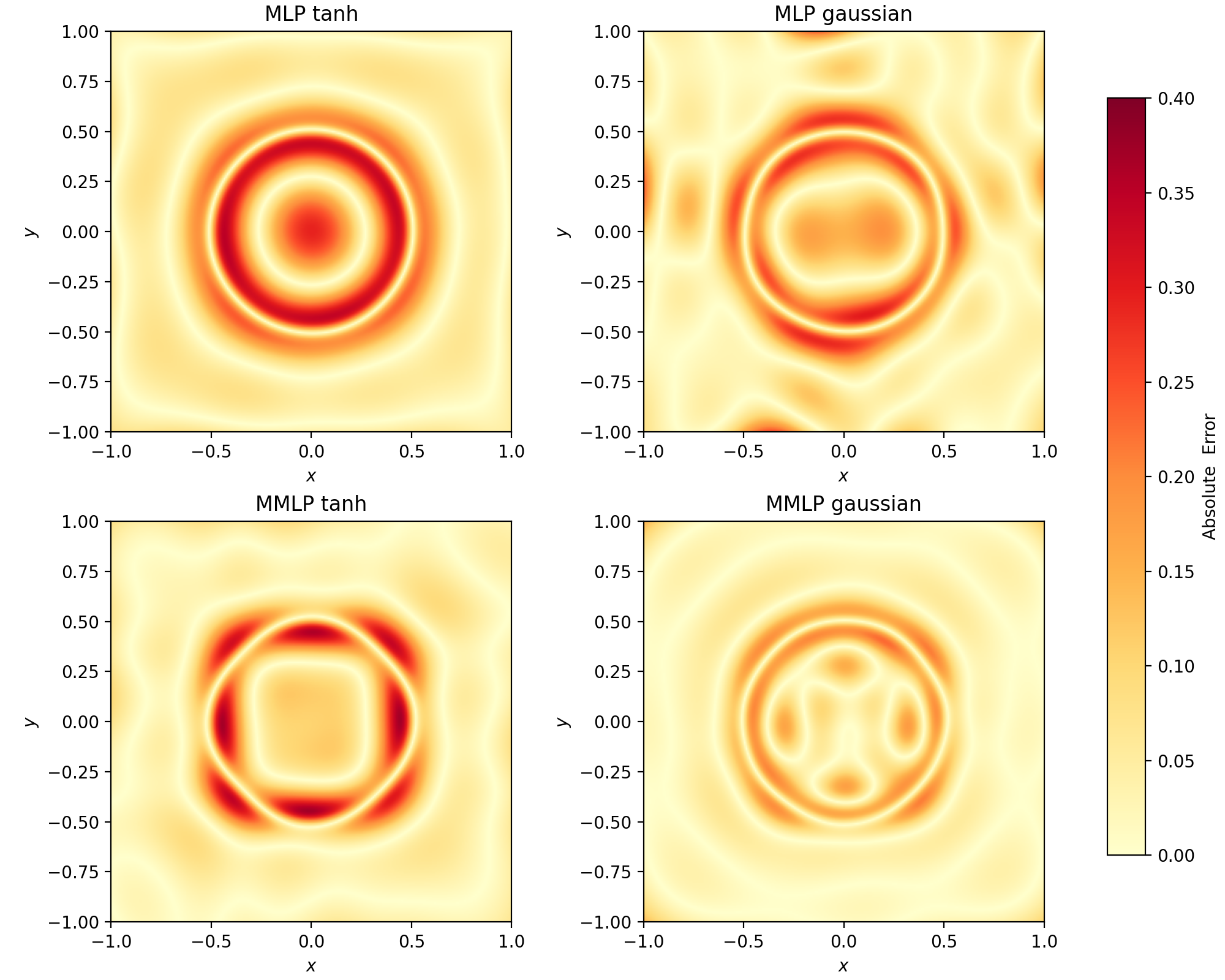}
    \caption{Figure 1: Absolute error for the mollified circle under the $L^2$ loss. 
    Top: MLP; bottom: MMLP; left: tanh; right: gaussian.
    The MMLP produces a more structured and geometrically aligned error pattern, whereas the MLP exhibits more spatially diffuse residuals away from the transition region.
    }
    \label{fig:circle:absErr}
\end{figure}

\subsubsection{Cone: point-type localization}
We next consider the cone example, where Figure \ref{fig:cone:absErr} displays the absolute error distributions for both architectures.

As in the previous example, the dominant approximation error is localized near the region of reduced regularity, here the origin. However, differences in the residual structure are again apparent, particularly for the Gaussian activation.
\begin{itemize}
    \item Concentration near the point of reduced regularity.
    Both architectures exhibit a clear error peak near the origin, reflecting the intrinsic difficulty of approximating the cone near this point.
    \item Structure of residual errors.
    Away from the origin, the MLP tends to produce more spatially dispersed residual errors, whereas the MMLP yields residuals that remain more localized near the origin and exhibit a more organized spatial pattern, with error structures that appear aligned with the coordinate directions. 
    \item Consistency across singular geometries.
    Despite the different regularity structures of the mollified circle and the cone, the same qualitative distinction between additive and multiplicative architectures persists.
\end{itemize}
The numerical results for both examples are consistent with the theoretical analysis, suggesting that the observed localization behavior depends on the network architecture and is not specific to a particular target geometry.
In particular, the numerical experiments highlight differences in the \emph{spatial structure} of approximation errors, describing where residual errors concentrate and how their geometry varies across architectures.

\begin{figure}[ht]\centering
    \includegraphics[width=0.7\linewidth]{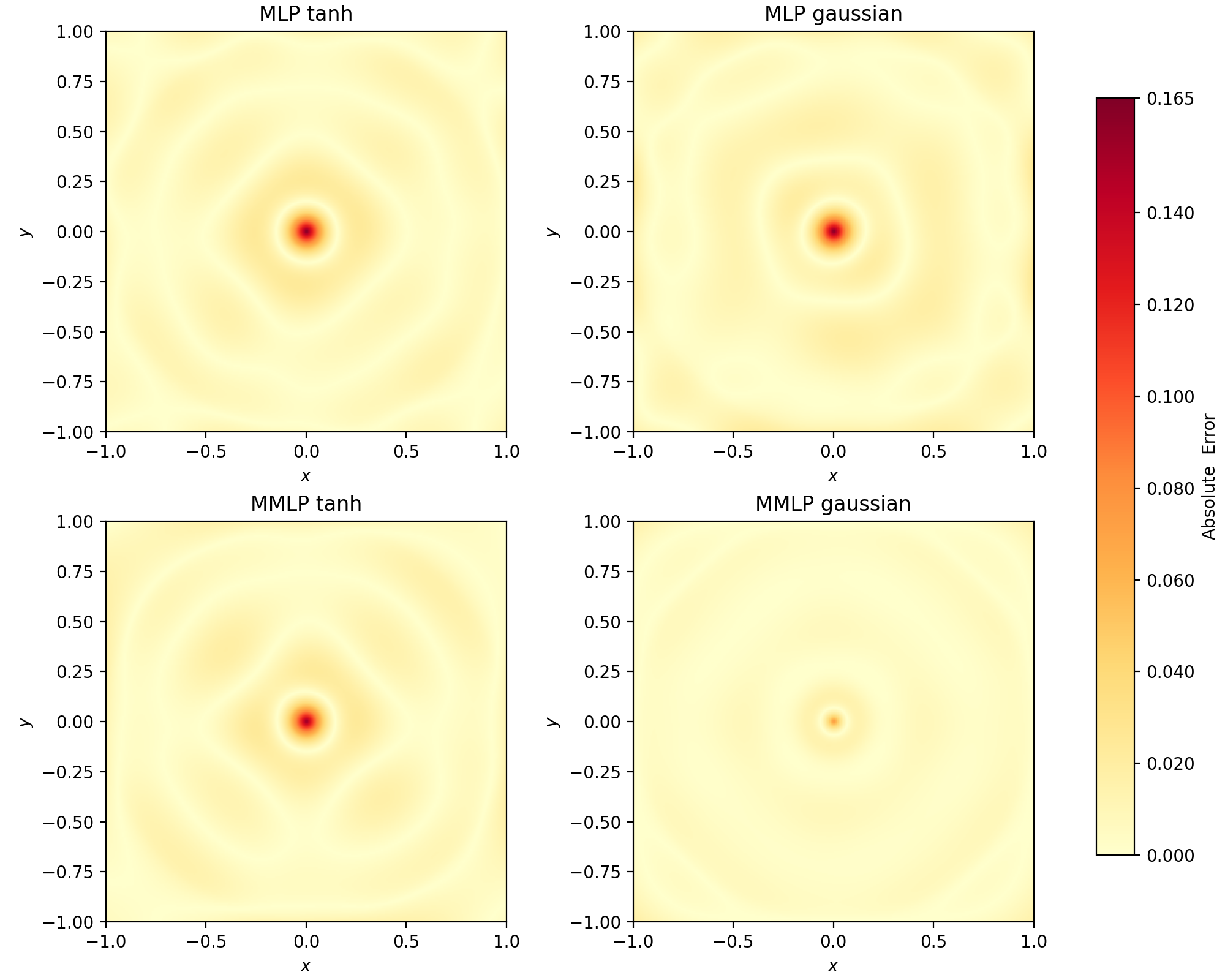}
    \caption{Figure 2: Absolute error for the cone target under the $L^2$ loss.
    Top: MLP; bottom: MMLP; left: tanh; right: gaussian.
    The MMLP produces residual errors that are more tightly localized and geometrically organized around the origin, whereas the MLP exhibits more spatially dispersed error patterns.
    }
\label{fig:cone:absErr}
\end{figure}

\subsection{Regularity-sensitive analysis via Zygmund-type seminorms}
\label{sec:numeric:zygmund}

The observations based on absolute error distributions in Section~\ref{sec:numeric:error_localization} describe where approximation errors occur and how they are spatially distributed, without addressing their behavior at small scales or their regularity properties.

In contrast, the present subsection focuses on regularity-sensitive quantities that capture the small-scale behavior of approximation errors.
In particular, we consider Zygmund-type seminorms, which quantify pointwise oscillations through second-order difference quotients and are directly related to the Sobolev--Zygmund embedding discussed in Sections~\ref{sec:intro} and~\ref{sec:theory}.

\emphTitle{Zygmund seminorms and pointwise oscillation}
Zygmund-type seminorms characterize regularity through second-order difference quotients.
At an informal level, they control expressions of the form
\begin{align}
    \sup_{x \in \Omega} \;
    \sup_{h \neq 0}
    \frac{\left| u(x+h) + u(x-h) - 2u(x) \right|}{|h|^{1+\alpha}} .
\end{align}
and thus quantify \emph{maximum pointwise oscillations} rather than averaged error magnitudes.
Convergence of such seminorms therefore indicates not merely a reduction in global approximation error, but an improvement in the stability of the approximation with respect to regularity variations.

In the numerical experiments below, we fix $\alpha=0.8$ and report errors in the corresponding $\C$-type Zygmund seminorm.
This choice is not essential: any exponent $\alpha\in(0,1)$ would lead to qualitatively similar behavior.
The value $\alpha=0.8$ is chosen as a representative interior exponent that balances sensitivity to oscillations with numerical stability.

\emphTitle{Cone example: $\L$ training}
We first consider the cone example trained using the $\L$ loss. This target exhibits a pointwise loss of second-order regularity at the origin, making it a natural test case for regularity-sensitive metrics.
Figure~\ref{fig:zyg:L2} shows the evolution of the $\H$-type error (top) and the Zygmund seminorm error with $\alpha=0.8$ (bottom).

For the additive architecture (MLP; the left panel), the Zygmund seminorm error decreases initially but then exhibits a pronounced slowdown, indicating an early saturation of regularity-sensitive convergence.
By contrast, the multiplicative architecture (MMLP; the right panel) displays a more sustained decay of the Zygmund seminorm over training iterations, suggesting improved stability with respect to maximum pointwise oscillations.
Notably, these differences cannot be inferred from the corresponding \(\H\)-type error curves alone, highlighting the complementary role of Zygmund-type seminorms as regularity-sensitive metrics.

\emphTitle{Cone example: $\H$-type training}
We next repeat the same experiment using the $\H$-type loss, which incorporates second-order information through a Laplacian-based penalty.
Figure~\ref{fig:zyg:H2} reports the corresponding $\H$-type and Zygmund seminorm errors.

Including second-order regularity information in the loss mitigates saturation effects for both architectures.
Nevertheless, the difference in convergence behavior observed under \(\L\) training persists:
the multiplicative architecture continues to exhibit a more stable and sustained decay of the Zygmund seminorm, while the additive architecture shows slower convergence in regularity-sensitive quantities.
This indicates that the observed behavior is not an artifact of the particular loss function, but reflects an underlying architectural effect.
\\

Taken together, the cone example illustrates that convergence of Zygmund-type seminorms captures aspects of approximation behavior that are not fully reflected by global error measures alone.
While both loss choices lead to comparable trends in $\L$- or $\H$-type errors, the evolution of regularity-sensitive quantities differs markedly between additive and multiplicative architectures.
These observations are consistent with the theoretical analysis of Sections~\ref{sec:intro} and~\ref{sec:theory}, where multiplicative constructions are associated with localized approximation aligned with pointwise regularity control.

\begin{figure}[ht]\centering
    \includegraphics[width=0.9\linewidth]{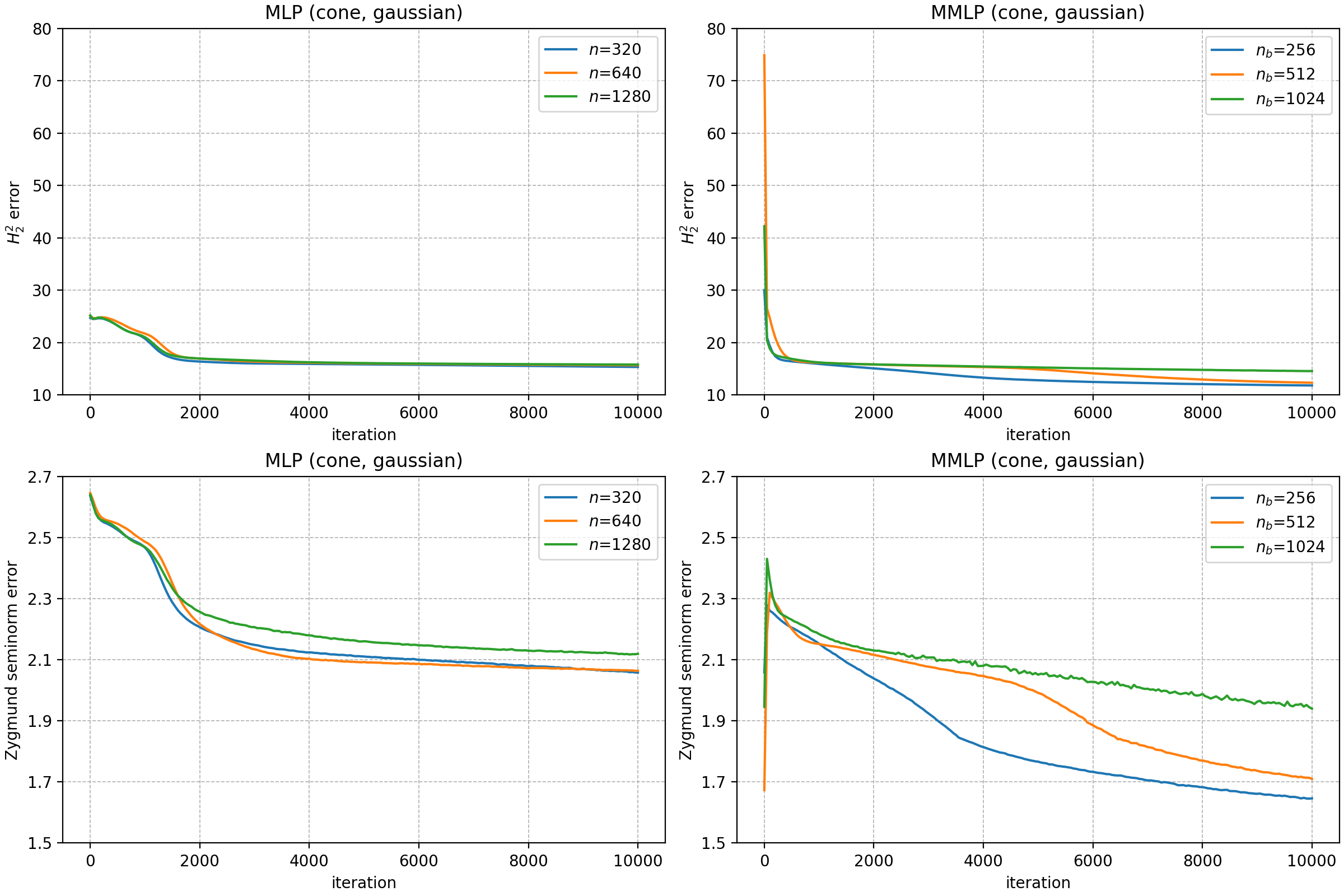}
    \caption{
    Convergence of regularity-sensitive error measures for the cone target under the $\L$ loss.
    Top: $\H$-type error; bottom: Zygmund seminorm $\C$ error.
    The MMLP exhibits a more sustained decay of the Zygmund seminorm, whereas the MLP shows earlier saturation.
    }
    \label{fig:zyg:L2}
\end{figure}

\begin{figure}[ht]\centering
    \includegraphics[width=0.9\linewidth]{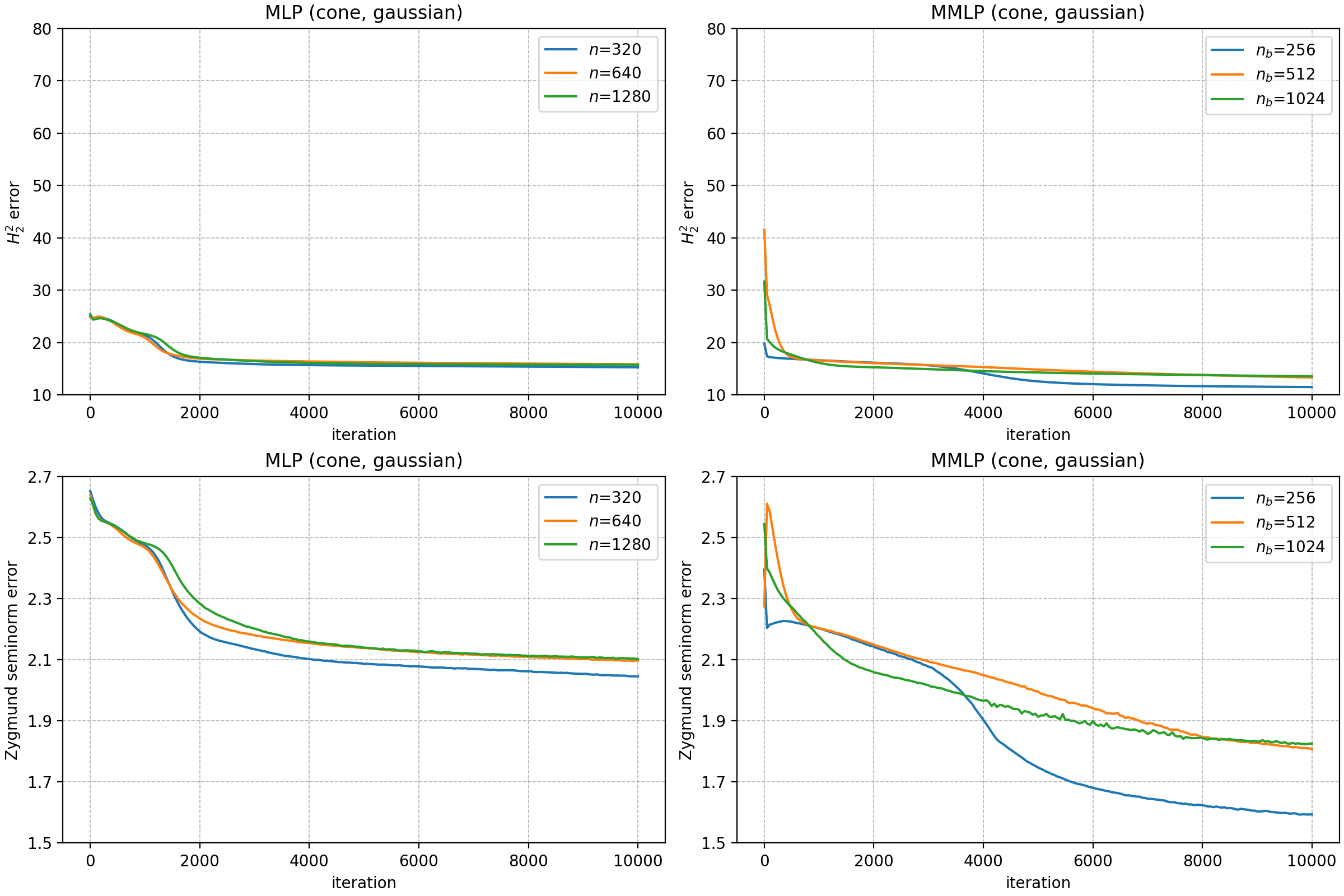}
    \caption{Convergence of regularity-sensitive error measures for the cone target under the $\H$ loss.
    Top: $\H$-type error; bottom: Zygmund seminorm $\C$ error.
    The MMLP exhibits a more sustained decay of the Zygmund seminorm, whereas the MLP shows earlier saturation.
    }
    \label{fig:zyg:H2}
\end{figure}

\appendix
\section{Architecture table}
\label{appendix:architecture}

Appendix \ref{appendix:architecture} provides a detailed summary of the network architectures and training configurations used in the numerical experiments. Table \ref{tab:arch_config} reports the exact neuron (or multiplicative block) configurations and the corresponding number of trainable parameters for both the additive (MLP) and multiplicative (MMLP) architectures.
All remaining training settings—including the optimizer, learning rate, number of iterations, activation functions, and batch size—are kept identical across architectures, ensuring a fair and controlled comparison. Detailed description on `$h$ for $\H$-type loss' is given in Appendix \ref{appendix:loss discretization}. 

\begin{table}[ht]\centering
\label{tab:arch_config}
\begin{tabular}{lcc}
\hline
\textbf{Configuration} 
& \textbf{MLP} 
& \textbf{MMLP} \\
\hline
Architecture type 
& Additive (MLP) 
& Multiplicative (MMLP) \\

Number of hidden layer
& 1 & 1\\

Neuron / block configuration 
& $n = (a) 320 \, (b) 640 \, (c)1280$ 
& $n_b = (a)256 \, (b)512 \, (c)1024$ \\

Total trainable parameters 
&  \multicolumn{2}{c}{$(a) 1281 \, (b) 2561 \, (c) 5121$}\\

Activation functions 
& \multicolumn{2}{c}{tanh, gaussian} \\

Optimizer 
& \multicolumn{2}{c}{Adam} \\

Learning rate 
& \multicolumn{2}{c}{$10^{-3}$} \\

Number of training iterations 
& \multicolumn{2}{c}{$10{,}000$} \\

Number of training samples
& \multicolumn{2}{c}{$50{,}000$} \\

Batch size 
& \multicolumn{2}{c}{$2{,}048$} \\

$h$ for $\H$-type loss:   
& \multicolumn{2}{c}{$1/128$} \\
\hline
\end{tabular}
\caption{Summary of network architectures and training configurations used in the numerical experiments. 
Except for the network architecture (MLP vs.\ MMLP) and the corresponding neuron counts, all training settings are kept identical.}
\end{table}

\section{Discretization of loss functions in implementation}
\label{appendix:loss discretization}

This appendix describes the numerical implementation of the loss functions used in the training procedure.
While the $\L$ loss follows a standard discretization, the $\H$-type loss involves a finite-difference approximation of second-order derivatives and therefore requires additional explanation.

\subsection{$\L$ loss}
Let $\Omega \subset \mathbb{R}^2$ denote the computational domain and let $\{x_i\}_{i=1}^M \subset \Omega$ be a uniform Cartesian grid used to discretize the loss.
The $\L$ loss is approximated by
\[
\mathcal{L}_{\L}(F)
=
\frac{1}{M}
\sum_{i=1}^M
\bigl|F(x_i) - f(x_i)\bigr|^2,
\]
where $F$ denotes the network output and $f$ the target function. This discretization is equivalent to the standard mean squared error (MSE) used in practice, up to a constant scaling factor. 

\subsection{$\H$-type loss}
To incorporate second-order regularity information, we consider an $\H$-type loss of the form
\[
\mathcal{L}_{\H}(F)
=
\mathcal{L}_{\L}(F)
+
\lambda \,
\mathcal{L}_{\Delta}(F),
\]
where
\[
\mathcal{L}_{\Delta}(F)
=
\frac{1}{M}
\sum_{i=1}^M
\bigl|
\Delta_h F(x_i) - \Delta_h f(x_i)
\bigr|^2,
\]
and $\Delta_h$ denotes a finite-difference approximation of the Laplacian.

\emphTitle{Finite-difference approximation}
The Laplacian $\Delta f$ is approximated using the standard second-order central finite-difference stencil on the uniform grid.
Let $h>0$ denote the grid spacing.
For interior points, the discrete Laplacian is given by
\[
\Delta_h f(x_{i,j})
=
\frac{
f(x_{i+1,j}) + f(x_{i-1,j}) + f(x_{i,j+1}) + f(x_{i,j-1}) - 4f(x_{i,j})
}{h^2}.
\]
Boundary points are treated using the same stencil, with target function values evaluated directly on the grid.

\emphTitle{Choice of the coefficient $\lambda$ and its effect on training behavior}
The coefficient $\lambda>0$ controls the relative weight of the Laplacian matching term in the $\H$-type loss.
In the experiments reported in this paper, $\lambda$ is fixed to a constant value across all runs.
This choice reflects the diagnostic role of the $\H$-type loss: the goal is not to optimize $\lambda$, but to examine how incorporating second-order information influences the convergence of regularity-sensitive quantities.

The effect of $\lambda$ must be interpreted in relation to the grid spacing $h$ used in the finite-difference approximation.
Since the discrete Laplacian $\Delta_h$ scales like $h^{-2}$, the Laplacian-based term amplifies high-frequency components relative to the $\L$ term.
Consequently, a different choice of $h$ would require a corresponding rescaling of $\lambda$ in order to maintain a comparable balance between the two contributions to the loss.

In the numerical experiments, the grid spacing is fixed to $h=1/128$, and the penalty parameter is set to $\lambda=10^{-2}$.
With this choice, the Laplacian matching term contributes at a scale that is comparable to, but does not dominate, the $\L$ term.
Empirically, this setting is sufficient to mitigate saturation effects in regularity-sensitive metrics while avoiding numerical instability or excessively slow convergence.

If $\lambda$ is chosen too small, the $\H$-type loss reduces effectively to the $\L$ loss and provides little additional regularity control.
Conversely, if $\lambda$ is chosen too large, the Laplacian term may dominate the optimization, leading to slower convergence or unstable training dynamics.
A systematic study of the optimal choice of $\lambda$ is beyond the scope of this work.
Here, $\lambda$ is fixed to a representative value in order to isolate qualitative effects of incorporating second-order information, rather than to achieve optimal performance.

\section{Discretization of Zygmund-type seminorms}
\label{appendix:zygmund}

This appendix describes the numerical evaluation of Zygmund-type seminorms used in the regularity-sensitive analysis of Section \ref{sec:numeric:zygmund}. 
Unlike the loss functions discussed in Appendix \ref{appendix:loss discretization}, the Zygmund seminorm is \emph{not} used for training, but solely as an evaluation quantity to assess pointwise regularity of the approximation.

\emphTitle{Continuous definition}

Let $F:\Omega\subset\mathbb{R}^2\to\mathbb{R}$ be a sufficiently regular function and let $\alpha\in(0,1)$.
The Zygmund seminorm corresponding to $C^{0,\alpha}$ regularity is defined by
\[
[F]_{C^{0,\alpha}}
=
\sup_{x\in\Omega}\;
\sup_{0<|h|\le h_{\max}}
\frac{|F(x+h)+F(x-h)-2F(x)|}{|h|^{\alpha}},
\]
where $h\in\mathbb{R}^2$ denotes a continuous increment vector and $h_{\max}$ is chosen so that $x\pm h\in\Omega$.

This seminorm measures pointwise oscillations through second-order difference quotients and is sensitive to local irregularities that are not captured by global Sobolev norms.

\emphTitle{Discrete approximation}

In the numerical experiments, the Zygmund seminorm is evaluated on a uniform Cartesian grid with spacing $h^z$.
For each grid point $x_{i,j}$ and for a prescribed set of discrete increments $h^z_k$, we compute
\[
\frac{
\bigl|F(x_{i,j}+h^z_k)+F(x_{i,j}-h^z_k)-2F(x_{i,j})\bigr|
}{
|h^z_k|^{\alpha}
},
\]
and take the maximum over all admissible increments and grid points.

In practice, the increments $h^z_k$ are chosen to align with the coordinate directions and their multiples, consistent with the grid resolution.
Boundary points are handled by restricting to increments for which all required grid values are available.

\emphTitle{Role of the grid spacing}

The discrete Zygmund seminorm depends explicitly on the grid spacing $h^z$, as the smallest admissible increment is of order $h^z$.
As $h^z\to 0$, the discrete seminorm provides an increasingly accurate approximation of its continuous counterpart.
In the present study, $h^z$ is fixed by the grid resolution, and all Zygmund seminorms are evaluated consistently on the same grid.

Since the seminorm is used only for comparative and qualitative analysis, no attempt is made to extrapolate its value to the continuum limit.
Instead, convergence and saturation behavior are assessed relative to the fixed discretization.

\emphTitle{Interpretation}

A decrease in the Zygmund seminorm during training indicates improved control of maximum pointwise oscillations and enhanced regularity of the approximation.
Conversely, saturation of the seminorm suggests that further reduction of global error does not translate into improved small-scale behavior.
For this reason, Zygmund-type seminorms provide a complementary diagnostic to $\L$- and $\H$-type error measures.

\bibliographystyle{plain}

\end{document}